\documentclass[reqno]{amsart}
\usepackage{amsmath}%
\usepackage{amsfonts}%
\usepackage{amssymb}%
\usepackage{graphicx}
\usepackage{mathrsfs}
\usepackage{hyperref}
\usepackage{fullpage}
\newtheorem{theorem}{Theorem}
\theoremstyle{plain}

\newtheorem{conjecture}[theorem]{Conjecture}

\newtheorem{fact}[theorem]{Fact}
\newtheorem{lemma}[theorem]{Lemma}

\newtheorem{proposition}[theorem]{Proposition}

\numberwithin{equation}{section}

\def\F{\mathcal{F}}

\def\HM{\mathcal{HM}}
\def\K{\mathcal{K}}

\def\M{\mathcal{M}}

\begin{document}
\title{A degree version of the Hilton--Milner theorem}
\thanks{
JH is supported by FAPESP (Proc. 2013/03447-6, 2014/18641-5). HH is partially supported by a Simons Collaboration Grant.
YZ is partially supported by NSF grant DMS-1400073.}
\author{Peter Frankl}
\address{Alfr\'ed R\'enyi Institute of Mathematics, P.O.Box 127, H-1364 Budapest, Hungary}
\email[Peter Frankl]{peter.frankl@gmail.com}
\author{Jie Han}
\address{Instituto de Matem\'{a}tica e Estat\'{\i}stica, Universidade de S\~{a}o Paulo, Rua do Mat\~{a}o 1010, 05508-090, S\~{a}o Paulo, Brazil}
\email[Jie Han]{jhan@ime.usp.br}
\author{Hao Huang}
\address{Department of Math and CS, Emory University, Atlanta, GA 30322}
\email[Hao Huang]{hao.huang@emory.edu}
\author{Yi Zhao}
\address
{Department of Mathematics and Statistics, Georgia State University, Atlanta, GA 30303} 
\email[Yi Zhao]{yzhao6@gsu.edu}

\date{\today}
\subjclass[2010]{05D05}%
\keywords{Intersecting families, Hilton--Milner theorem, Erd\H{o}s--Ko--Rado theorem}

\begin{abstract}
An intersecting family of sets is trivial if all of its members share a common element. Hilton and Milner proved a strong stability result for the celebrated Erd\H{o}s--Ko--Rado theorem: when $n> 2k$, every non-trivial intersecting family of $k$-subsets of $[n]$ has at most $\binom{n-1}{k-1}-\binom{n-k-1}{k-1}+1$ members. One extremal family $\HM_{n, k}$ consists of a $k$-set $S$ and all $k$-subsets of $[n]$ containing a fixed element $x\not\in S$ and at least one element of $S$. We prove a degree version of the Hilton--Milner theorem: if $n=\Omega(k^2)$ and $\F$ is a non-trivial intersecting family of $k$-subsets of $[n]$, then $\delta(F)\le \delta(\HM_{n.k})$, where $\delta(\F)$ denotes the minimum (vertex) degree of $\F$.  Our proof uses several fundamental results in extremal set theory, the concept of kernels, and a new variant of the Erd\H{o}s--Ko--Rado theorem.
\end{abstract}

\maketitle

\section{Introduction}

A family $\F$ of sets is called \emph{intersecting} if $A\cap B\neq \emptyset$ for all $A, B\in \F$. A fundamental problem in extremal set theory is to study the properties of intersecting families. 
For positive integers $k, n$, let $[n]=\{1,2,\dots, n\}$ and $\binom{V}k$ denote the family of all $k$-element subsets ($k$-subsets) of $V$. 
We call a family on $V$ \emph{$k$-uniform} if it is a subfamily of $\binom{V}k$.
A \emph{full star} is a family that consists of all the $k$-subsets of $[n]$ that contains a fixed element. 
We call an intersecting family $\F$ \emph{trivial} if it is a subfamily of a full star. 
The celebrated Erd\H{o}s--Ko--Rado (EKR) theorem~\cite{EKR} states that, when $n\ge 2k$, every $k$-uniform intersecting family on $[n]$ has at most $\binom{n-1}{k-1}$ members, and the full star shows that the bound $\binom{n-1}{k-1}$ is best possible.
Hilton and Milner~\cite{HM67} proved the uniqueness of the extremal family in a stronger sense: if $n> 2k$, every non-trivial intersecting family of $k$-subsets of $[n]$ has at most $\binom{n-1}{k-1}-\binom{n-k-1}{k-1}+1$ members. It is easy to see that the equality holds for the following family, denoted by $\HM_{n, k}$, which consists of a $k$-set $S$ and all $k$-subsets of $[n]$ containing a fixed element $x\not\in S$ and at least one vertex of $S$. For more results on intersecting families, see a recent survey by Frankl and Tokushige~\cite{FrTo16}.

Given a family $\F$ and $x\in V(\F)$, we denote by $\F(x)$ the subfamily of $\F$ consisting of all the members of $\F$ that contain $x$, i.e., $\F(x):=\{F\in \mathcal{\F}: x\in F\}$. Let $d_{\F}(x) := |\F(x)|$ be the \emph{degree} of $x$. Let $\Delta(\F) := \max_x d_{\F}(x)$ and $ \delta(\F) := \min_x d_{\F}(x)$ denote the maximum and minimum degree of $\F$, respectively. There were extremal problems in set theory that considered the maximum or minimum degree of families satisfying certain properties. For example, Frankl \cite{Frankl87} extended the Hilton--Milner theorem by giving sharp upper bounds on the size of intersecting families with certain maximum degree. Bollob\'as, Daykin, and Erd\H os \cite{BDE} studied the minimum degree version of a well-known conjecture of Erd\H os \cite{erdos65} on matchings.

Huang and Zhao~\cite{HuZh} recently proved a minimum degree version of the EKR theorem, which states that, if $n> 2k$ and $\F$ is a $k$-uniform intersecting family on $[n]$, then $\delta(\F)\le \binom{n-2}{k-2}$, and the equality holds only if $\F$ is a full star.
This result implies the EKR theorem immediately: given a $k$-uniform intersecting family $\F$,
by recursively deleting elements with the smallest degree until $2k$ elements are left, we derive that
\[
|\F|\le \binom{n-2}{k-2}+\binom{n-3}{k-2}+\cdots+\binom{2k-1}{k-2} +\binom{2k-1}{k-1} = \binom{n-1}{k-1}.
\]
Frankl and Tokushige~\cite{FrTo16_degree} gave a different proof of the result of \cite{HuZh} for $n\ge 3k$. Generally speaking, a minimum degree condition forces the sets of a family to be distributed somewhat evenly and thus the size of a family that is required to satisfy a property might be smaller than the one without degree condition. Unless the extremal family is very regular, an extremal problem under the minimum degree condition seems harder than the original extremal problem because one cannot directly apply the \emph{shifting method} (a powerful tool in extremal set theory).

In this paper we study the minimum degree version of the Hilton--Milner theorem.

\begin{theorem}\label{thm:main}
Suppose $k\ge 4$ and $n\ge c k^2$, where $c=30$ for $k=4,5$ and $c=4$ for $k\ge 6$. If $\F\subseteq \binom{[n]}{k}$ is a non-trivial intersecting family, then $\delta(\F)\le \delta(\HM_{n, k}) = \binom{n-2}{k-2} - \binom{n- k-2}{k-2}$.
\end{theorem}

Han and Kohayakawa~\cite{HaKo} recently determined the maximum size of a non-trivial intersecting family that is not a subfamily of $\HM_{n, k}$, which is $\binom{n-1}{k-1} - \binom{n-k-1}{k-1} - \binom{n-k-2}{k-2}+2$.
Later Kostochka and Mubayi~\cite{KoMu} determined the maximum size of a non-trivial intersecting family that is not a subfamily of $\HM_{n, k}$ or the extremal families given in \cite{HaKo} for sufficiently large $n$.
Furthermore, Kostochka and Mubayi~\cite[Theorem 8]{KoMu} characterized \emph{all} maximal intersecting $3$-uniform families $\F$ on $[n]$ for $n\ge 7$ and $|\F|\ge 11$.
Using a different approach, Polcyn and Ruci\'nski~\cite[Theorem 4]{PoRu} characterized \emph{all} maximal intersecting $3$-uniform families $\F$ on $[n]$ for $n\ge 7$, in particular, there are fifteen such families, including the full star and $\HM_{n, 3}$. It is straightforward to check that all these families have minimum degree at most $3$ -- this gives the following proposition.

\begin{proposition}\label{thm:k=3}
If $n\ge 7$ and $\F\subseteq \binom{[n]}{3}$ is a non-trivial intersecting family, then $\delta(\F)\le \delta(\HM_{n, 3}) = 3$.
\end{proposition}

In order to prove Theorem~\ref{thm:main}, we prove a new variant of the EKR theorem, which is closely related to the EKR theorem for direct products given by Frankl (see Theorem~\ref{thm:Frankl96}).

\begin{theorem}	\label{lem_three_sets}
Given integers $k\ge 3$, $\ell\ge 4$, and $m \ge k\ell$, let $T_1, T_2, T_3$ be three disjoint $\ell$-subsets of $[m]$. If $\mathcal{F}$ is a $k$-uniform intersecting family on $[m]$ such that every member intersects all of $T_1, T_2, T_3$, then $|\mathcal{F}| \le \ell^2 \binom{m-3}{k-3}$.
\end{theorem}
Theorem~\ref{lem_three_sets} becomes trivial when $\ell=1$ because every family $\F$ of $k$-sets that intersect  $T_1, T_2, T_3$ satisfies $|\F|\le \binom{m-3}{k-3}$. Our bound in Theorem~\ref{lem_three_sets} is asymptotically tight because a star with a center in $T_1\cup T_2\cup T_3$ contains about $\ell^2 \binom{m-3}{k-3}$ $k$-sets that intersect  $T_1, T_2, T_3$.

It was shown in \cite{HuZh} that one can derive the minimum degree version of the EKR theorem for $n=\Omega(k^2)$ by using the Hilton--Milner Theorem and simple averaging arguments (thus the difficulty of the result in \cite{HuZh} lies in deriving the tight bound $n\ge 2k+1$).
However, we can not use this naive approach to prove Theorem~\ref{thm:main} for sufficiently large $n$. Indeed, let $\F$ be a non-trivial intersecting family that is not a subfamily of $\HM_{n,k}$. The result of Han and Kohayakawa~\cite{HaKo} says that $|\F|$ is asymptotically at most $(k-1)\binom{n-2}{k-2}$, and in turn, the average degree of $\F$ is asymptotically at most $\frac{k(k-1)}{k-2}\binom{n-3}{k-3}$. Unfortunately, this is much larger than $\delta(\HM_{n,k}) \approx k\binom{n-3}{k-3}$ as $k$ is fixed and $n$ is sufficiently large.

Our proof of Theorem~\ref{thm:main} applies several fundamental results in extremal set theory as well as Theorem~\ref{lem_three_sets}. The following is an outline of our proof.
Let $\F$ be a non-trivial intersecting family such that $\delta(\F)> \delta(\HM_{n, k})$.
For every $u\in [n]$, we obtain a lower bound for $|\F\setminus \F(u)|$ by applying the assumption on $\delta(\F)$ and the Frankl--Wilson theorem \cite{Frankl78,Wilson} on the maximum size of $t$-intersecting families. If $k=4,5$, then we derive a contradiction by considering the \emph{kernel} of $\F$ (a concept introduced by Frankl \cite{Frankl78_kernel}). When $k\ge 6$, we separate two cases based on $\Delta(\F)$. When $\Delta(\F)$ is large, assume that $|\F(u)| = \Delta(\F)$ and let $\F_2 :=  \F\setminus \F(u)$. A result of Frankl \cite{Frankl13} implies that $\F(u)$ contains three edges $E_i:=\{u\}\cup T_i$, $i\in [3]$, where $T_1, T_2, T_3$ are pairwise disjoint.
Since $\F_2$ is intersecting and every member of $\F_2$ meets each of $T_1, T_2, T_3$, Theorem~\ref{lem_three_sets} gives an upper bound on $|\F_2|$, which contradicts the lower bound that we derived earlier. When $\Delta(\F)$ is small, we apply the aforementioned result of Frankl~\cite{Frankl87} to obtain an upper bound on $|\F|$, which contradicts the assumption on $\delta(\F)$.

\section{Tools}
\subsection{Results that we need}
%
%
Given a positive integer $t$, a family $\F$ of sets is called \emph{$t$-intersecting} if $|A\cap B|\ge t$ for all $A, B\in \F$. 
A $t$-intersecting EKR theorem was proved in \cite{EKR} for sufficiently large $n$.
Later Frankl~\cite{Frankl78} (for $t\ge 15$) and Wilson~\cite{Wilson} (for all $t$) determined the exact threshold for $n$.

\begin{theorem}\cite{Frankl78, Wilson}\label{thm:Wilson}
Let $n\ge (t+1)(k-t+1)$ and let $\F$ be a $k$-uniform $t$-intersecting families on $[n]$.
Then $|\F|\le \binom{n-t}{k-t}$.
\end{theorem}

As mentioned in Section 1, Frankl~\cite{Frankl87} determined the maximum possible size of an intersecting family under a maximum degree condition.

\begin{theorem}\cite{Frankl87}\label{thm:Frankl}
Suppose $n > 2k$, $3\le i\le k+1$, $\F\subseteq \binom{[n]}k$ is intersecting.
If $\Delta(\F)\le \binom{n-1}{k-1} - \binom{n-i}{k-1}$, then $|\F|\le \binom{n-1}{k-1} - \binom{n-i}{k-1} + \binom{n-i}{k-i+1}$.
\end{theorem}

Given a $k$-uniform family $\F$, a \emph{matching of size $s$} is a collection of $s$ vertex-disjoint sets of $\F$.
A well-known conjecture of Erd\H os \cite{erdos65} states that if $n\ge (s+1)k$ and $\F\subseteq \binom{[n]}k$ satisfies $|\F| > \max\{\binom{n}{k} - \binom{n-s}{k}, \binom{k(s+1)-1}{k}\}$, then $\F$ contains a matching of size $s+1$.
Frankl~\cite{Frankl13} verified this conjecture for $n\ge (2s+1)k-s$.

\begin{theorem}\cite{Frankl13}\label{thm:Frankl13}
Let $n\ge (2s+1)k-s$ and let $\F\subseteq \binom{[n]}k$. If $|\F| > \binom{n}{k} - \binom{n-s}{k}$, then $\F$ contains a matching of size $s+1$.
\end{theorem}

Frankl~\cite{Frankl96} proved an EKR theorem for direct products. 

\begin{theorem} \cite{Frankl96} \label{thm:Frankl96} 
Suppose $n=n_1+\cdots+n_d$ and $k=k_1 + \cdots + k_d$, where $n_i\ge k_i$ are positive integers.  Let $X_1\cup \cdots \cup X_d$ be a partition of $[n]$ with $|X_i|=n_i$, and
	\[ \mathcal{H}=\left\{F \in \binom{[n]}{k}: |F \cap X_i|=k_i~\textrm{for~}i=1, \dots, d \right\}. \]
	If $n_i \ge 2k_i$ for all $i$ and $\mathcal{F} \subseteq \mathcal{H}$ is intersecting, then 
	\[ \frac{|\mathcal{F}|}{|\mathcal{H}|} \le \max_{i} \frac{k_i}{n_i}. \]
\end{theorem}
Note that the $d=1$ case of Theorem~\ref{thm:Frankl96} is the EKR theorem.

\subsection{Kernels of intersecting families}

Frankl introduced the concept of \emph{kernels} (and called them \emph{bases}) for intersecting families in \cite{Frankl78_kernel}.
Given $\F\subseteq \binom{V}k$, a set $S\subseteq V$ is called a \emph{cover} of $\F$ if $S\cap A\ne \emptyset$ for all $A\in \F$. For example, if $\F$ is intersecting, then every member of $\F$ is a cover. 
Given an intersecting family $\F$, we define its \emph{kernel} $\K$ as
\[
\K:=\{S: \text{$S$ is a cover of $\F$ and any $S'\subsetneq S$ is not a cover of $\F$}\}.
\]

An intersecting family $\F$ is called \emph{maximal} if $\F\cup \{G\}$ is not intersecting for any $k$-set $G\not\in F$.
Note that, when proving Theorem~\ref{thm:main}, we may assume that $\F$ is maximal because otherwise we can add more $k$-sets to $\F$ such that the resulting intersecting family is still non-trivial and satisfies the minimum degree condition.
We observe the following fact on the kernels.

\begin{fact}\label{fact:kernel}
If $n\ge 2k$ and $\F\in \binom{[n]}k$ is a maximal intersecting family, then $\K$ is also intersecting.
\end{fact}

\begin{proof}
Suppose there are $K_1, K_2\in \K$ such that $K_1\cap K_2=\emptyset$.
Since $n\ge 2k$, we can find two disjoint $k$-sets $F_1, F_2$ on $[n]$ such that $K_i\subseteq F_i$ for $i=1,2$.
For $i=1,2$, since $K_i$ is a cover of $\F$, $F_i$ intersects all members of $\F$.
Since $\F$ is maximal, we derive that $F_1, F_2\in \F$. This contradicts the assumption that $F_1, F_2$ are disjoint.
\end{proof}

For $i\in [k]$, let $\K_i:=\K\cap \binom{[n]}{i}$. If an intersecting family $\F$ is non-trivial, then $\K_1=\emptyset$. Below we prove an upper bound for $|\K_i|$, $3\le i\le k$, where the $i=k$ case was given by Erd\H os and Lov\'asz \cite{ErLo75}.

\begin{lemma}\label{lem:Frankl}
For $3\le i\le k$, we have $|\K_i|\le k^i$.

\end{lemma}

In order to prove Lemma~\ref{lem:Frankl}, We use a result of H{\aa}stad, Jukna, and Pudl{\'a}k \cite[Lemma 3.4]{Hastad95}. 
Given a family $\F$, the \emph{cover number} of $\F$, denoted by $\tau(\F)$, is the size of the smallest cover of $\F$.

\begin{lemma}\cite{Hastad95}\label{lem:flower}
If $\F$ is an $i$-uniform family with $|\F| > k^i$, then there exists a set $Y$ such that $\tau(\F_Y) \ge k+1$, where $\F_Y := \{F\setminus Y: F\in \F, F\supseteq Y\}$.
\end{lemma}

\begin{proof}[Proof of Lemma~\ref{lem:Frankl}]
Suppose $|\K_i| > k^i$ for some $3\le i\le k$.
Then by Lemma~\ref{lem:flower}, there exists a set $Y$ such that $\tau((\K_i)_Y) \ge k+1$. In particular, $(\K_i)_Y$ is nonempty, namely, there exists $K\in \K_i$ such that $Y\subsetneq K$. By the definition of $\K$, this implies that $Y$ is not a cover of $\F$, so there exists $F\in \F$ such that $F\cap Y=\emptyset$.
Since each member of $\K_i$ is a cover of $\F$, each of them intersects $F$.
This implies that $\tau((\K_i)_Y)\le |F|=k$, a contradiction.
\end{proof}

\section{Proof of Theorem~\ref{lem_three_sets}}

In this section we derive Theorem~\ref{lem_three_sets} from Theorem~\ref{thm:Frankl96}.

\begin{proof}[Proof of Theorem~\ref{lem_three_sets}]
Let $\mathcal{F}_r$ consist of all the subsets of $\mathcal{F}$ that intersect with $T_1 \cup T_2 \cup T_3$ in exactly $r$ elements. 
Then $\mathcal{F}=\mathcal{F}_3 \cup \mathcal{F}_4 \cup \cdots \cup \mathcal{F}_k$.
Let $X_1=T_1$, $X_2=T_2$, $X_3=T_3$, $X_4=[m] \setminus (T_1 \cup T_2 \cup T_3)$, and $k_1=k_2=k_3=1$, $k_4=k-3$. Since $m\ge k\ell$, we have $1/\ell \ge (k-3)/(m-3\ell)$. Since $\ell\ge 2$, we can apply Theorem~\ref{thm:Frankl96} to conclude that 
\[
|\mathcal{F}_3| \le \ell^3 \binom{m-3\ell}{k-3} \cdot \frac{1}{\ell} = \ell^2 \binom{m-3\ell}{k-3}.
\]

Note that a set $S \in \mathcal{F}_4$ intersects $T_1, T_2, T_3$ with either $1, 1, 2$ or $1,2,1$ or $2,1,1$ elements. We partition $\mathcal{F}_4$ into three subfamilies accordingly. Our assumption implies 
\[
\frac{k-4}{m- 3\ell} \le \frac2{\ell}\le \frac12.
\]
We can apply Theorem~\ref{thm:Frankl96} to each subfamily of $\mathcal{F}_4$ and obtain that
\[
|\mathcal{F}_4| \le 3\binom{\ell}{2}\ell^2 \binom{m-3\ell}{k-4} \cdot \frac{2}{\ell}  = 3 (\ell-1) \ell^2 \binom{m-3\ell}{k-4}.
\]

Finally, for $5 \le r \le k$, we claim that $|\mathcal{F}_r| \le \ell^2 \binom{3\ell-3}{r-3} \binom{m-3\ell}{k-r}$.
Indeed, let $X_1=T_1 \cup T_2 \cup T_3$, $X_2=[m] \setminus X_1$, $k_1=r$ and $k_2=k-r$. Note that $|X_2|=m-3\ell \ge 2(k-r)$ and $r/(3\ell) \ge (k-r)/(m-3\ell)$.
If $|X_1|=3\ell \ge 2r$, then Theorem~\ref{thm:Frankl96} gives that
\[
|\mathcal{F}_r| \le  \binom{3\ell-1}{r-1} \binom{m-3\ell}{k-r} < \ell^2 \binom{3\ell-3}{r-3} \binom{m-3\ell}{k-r}.
\]
When $3\ell\le 2r$, we have $r\ge 6$ because $\ell \ge 4$. Hence,
\[
 \binom{3\ell}{r} < \frac{(3\ell)^3}{r(r-1)(r-2)} \binom{3\ell-3}{r-3} \le \frac{18\ell^2}{(r-1)(r-2)} \binom{3\ell-3}{r-3} < \ell^2 \binom{3\ell-3}{r-3},
\]
and the trivial bound on $|\F_r|$ gives that
\[
|\mathcal{F}_r| \le  \binom{3\ell}{r} \binom{m-3\ell}{k-r} < \ell^2 \binom{3\ell-3}{r-3} \binom{m-3\ell}{k-r}
\]
as claimed. Summing up the bounds for $|\F_3|, |\F_4|$ and $|\F_r|$ for $r\ge 5$, we have 
\begin{align*}
|\mathcal{F}| &=|\mathcal{F}_3|+|\mathcal{F}_4| + \sum_{r=5}^{k}  |\mathcal{F}_r|\\
& \le \ell^2 \binom{m-3\ell}{k-3} + 3(\ell-1)\ell^2 \binom{m-3\ell}{k-4} + \ell^2 \sum_{r=5}^k \binom{3\ell-3}{r-3} \binom{m-3\ell}{k-r} = \ell^2 \binom{m-3}{k-3},
\end{align*}
because $\binom{m-3}{k-3} = \sum_{i=0}^{k-3} \binom{m-3l}{k-3-i} \binom{3l-3}{i}$.
\end{proof}

\section{Proof of Theorem~\ref{thm:main}}

We start with some simple estimates.
First, for $n\ge c k^2$, $c\ge 1$ and $1\le t\le k-1$, we have
\begin{align}
\frac{\binom{n-2k+t-1}{k-2}}{\binom{n-t-1}{k-2}} &= \frac{ (n-2k+t-1) \cdots (n-3k+t+2) }{ (n-t-1) \cdots (n-t-k+2) } \ge \left(1 - \frac{2k-2t}{n-t-k+2} \right)^{k-2} \nonumber \\
& \ge 1 - \frac{ 2(k-t)(k-2) }{n -t -k+2} \ge \frac{c-2}{c}. \label{eq:k3}
\end{align}
Similarly, one can show that $\binom{n-k-2}{k-3} \ge \frac{c-1}{c} \binom{n-3}{k-3}$.
Second, if $\delta(\F) > \binom{n-2}{k-2} - \binom{n- k-2}{k-2}$, then we have
\begin{align}
|\F| &> \frac{n}{k}\left( \binom{n-2}{k-2} - \binom{n- k-2}{k-2} \right)> n\binom{n-k-2}{k-3} \nonumber \\
&\ge \frac{(c-1)n}{c} \binom{n-3}{k-3} > \frac{(c-1)}c (k-2) \binom{n-2}{k-2}. \label{eq:deg}
\end{align}

\begin{lemma}\label{lem:F2}
Suppose $k\ge 4$ and $n\ge 4 k^2$, $\F\subseteq \binom{[n]}{k}$ is a non-trivial intersecting family such that $\delta(\F)> \delta(\HM_{n, k}) = \binom{n-2}{k-2} - \binom{n- k-2}{k-2}$.
Then for any $u\in [n]$,
\begin{enumerate}
\item[($i$)] there exists $E, E'\in \F$ such that $u\notin E\cup E'$ and $|E\cap E'|=1$;
\item[($ii$)] $|\F\setminus \F(u)| > \frac{k-2}{2} \binom{n-2}{k-2}$.
\end{enumerate}
\end{lemma}

\begin{proof}
Given $u\in [n]$, write $\F_1= \F(u)$ and $\F_2 = \F\setminus \F_1$. If $|\F_2|=1$, then $\F\subseteq \HM_{n, k}$, and thus $\delta(\F)\le \delta(\HM_{n, k})$, a contradiction. So assume that $|\F_2|\ge 2$.

Let $t=\min |E \cap E'|$ among all distinct $E, E' \in \F_2$. Obviously $1 \le t \le k-1$, and $\F_2$ is a $t$-intersecting family on $[2,n]$.
Then since $n> 4k^2 \ge  (k-t+1)(t+1)+1$, we get $|\F_2| \le \binom{n-t-1}{k-t}$ by Theorem~\ref{thm:Wilson}.
Note that there exist $E, E'\in \F_2$ such that $|E\cap E'| = t$.
Since every set in $\F_1$ must intersect both $E$ and $E'$, for every $x\not\in E\cup E' \cup \{u\}$, by the inclusion-exclusion principle, we have
\begin{equation}\label{eq:F1x}
    |\F_1(x)| \le \binom{n-2}{k-2} - 2\binom{n-k-2}{k-2} + \binom{n-2k+t-2}{k-2}.
\end{equation}
Let $X= [n]\setminus (E\cup E'\cup \{u\} )$ and thus $|X| = n - 1 - (2k-t)$. Suppose $x\in X$ attains the minimum degree in $\F_2$ among all elements of $X$.
Since $|\F(x)|=|\F_1(x)|+|\F_2(x)|> \delta(\mathcal{HM}_{n, k})$, by~\eqref{eq:F1x} we have
\begin{equation}\label{eq:F2x}
|\F_2(x)| > \binom{n-k-2}{k-2} - \binom{n-2k+t-2}{k-2}. \nonumber
\end{equation}
By the definition of $x$ we get
\begin{align*}
|\F_2| &> \frac{|X|}{k-t} \left( \binom{n-k-2}{k-2} - \binom{n-2k+t-2}{k-2} \right) \ge \frac{|X|(k-t)}{k-t} \binom{n-2k+t-2}{k-3} \\
& = (k-2) \binom{n-2k+t-1}{k-2} ,
\end{align*}
where the factor $k-t$ comes from the fact that every member $F\in \F_2$ is counted at most $k-t$ times -- because $|F\cap E_1| \ge t$.
By~\eqref{eq:k3} with $c=4$ and $k\ge 4$, we get
\begin{equation*}\label{eq:F2}
|\F_2| > \frac{k-2}{2} \binom{n-t-1}{k-2} \ge \binom{n-t-1}{k-2},
\end{equation*}
which, together with $|\F_2| \le \binom{n-t-1}{k-t}$, implies that $t=1$, so ($i$) holds.
Since $t=1$, the first inequality above gives ($ii$).
\end{proof}

\begin{proof}[Proof of Theorem~\ref{thm:main}]

First assume that $k\ge 6$ and $n\ge 4 k^2$.
Suppose $\F\subseteq \binom{[n]}{k}$ is a non-trivial intersecting family such that $\delta(\F)> \delta(\HM_{n, k}) = \binom{n-2}{k-2} - \binom{n- k-2}{k-2}$.
Suppose $u\in [n]$ attains the maximum degree of $\F$ and write $\F': = \F\setminus \F(u)$.
If $|\F(u)|> \binom{n-1}{k-1} - \binom{n-3}{k-1}$, then by Theorem~\ref{thm:Frankl13}, the $(k-1)$-uniform family $\{E\setminus \{u\}: E\in \F(u)\}$ contains a matching $\M= \{ T_1, T_2, T_3 \}$ of size $3$. Every member of $\F'$ must intersect each of $T_1, T_2, T_3$.
By Theorem~\ref{lem_three_sets}, we have $|\F'| \le (k-1)^2\binom{n-4}{k-3}$. On the other hand, Lemma~\ref{lem:F2} Part~($ii$) implies that $|\F'| > \frac{k-2}{2} \binom{n-2}{k-2} = \frac{n-2}{2}\binom{n-3}{k-3} > 2(k-1)^2\binom{n-3}{k-3}$ because $n\ge 4k^2 \ge 4(k-1)^2+2$. This gives a contradiction.

We thus assume that $|\Delta(\F)|\le \binom{n-1}{k-1} - \binom{n-3}{k-1}$. By Theorem~\ref{thm:Frankl}, 
\[
|\F|\le \binom{n-1}{k-1} - \binom{n-3}{k-1} + \binom{n-3}{k-2} = \frac{3n - 2k -2}{n-2} \binom{n-2}{k-2} \le 3\binom{n-2}{k-2}.
\]
Since $\delta(\F) > \binom{n-2}{k-2} - \binom{n- k-2}{k-2}$, by~\eqref{eq:deg}, we have $|\F| > \frac{3}{4} (k-2) \binom{n-2}{k-2}$.
The upper and lower bounds for $|\F|$ together imply $k<6$, a contradiction.

\medskip

Now assume that $k=4,5$ and $n\ge 30 k^2$.
Since $\F$ is intersecting, each member of $\F$ is a cover of $\F$ and thus contains as a subset a minimal cover, which is a member of the kernel $\K$.
Thus $|\F|\le \sum_{i=1}^k |\K_i| \binom{n-i}{k-i}$. We know $\K_1= \emptyset$ because $\F$ is non-trivial. 
We observe that $|\K_2|\le 1$ -- otherwise assume $u v, u v'\in \K_2$ (recall that $\K_2$ is intersecting).
By the definition of $\K_2$, every $E\in  \F\setminus \F(u)$ contains both $v$ and $v'$ so every $E, E'\in  \F\setminus \F(u)$ satisfy that $|E\cap E'|\ge 2$, contradicting Lemma~\ref{lem:F2} Part~($i$).
By Lemma~\ref{lem:Frankl}, 
\[
 |\F| \le \binom{n-2}{k-2} + \sum_{i=3}^k k^i \binom{n-i}{k-i}.
\]
Since $n\ge 30 k^2$, for any $3\le i\le k$, we have
\[
k^{i-2} \binom{n-i}{k-i} = \binom{n-2}{k-2} \cdot k^{i-2} \cdot \frac{k-2}{n-2} \cdot \frac{k-3}{n-3} \cdots \frac{k-i+1}{n-i+1} \le \binom{n-2}{k-2} \frac1{30^{i-2}}.
\]
Thus
\[
 |\F| \le \binom{n-2}{k-2} + k^2\binom{n-2}{k-2}\sum_{i=3}^k \frac1{30^{i-2}}\le \binom{n-2}{k-2} \left(1 + \frac{k^2}{29}\right).
\]
On the other hand, by~\eqref{eq:deg}, we have $|\F| > \frac{29}{30} (k-2) \binom{n-2}{k-2} > \frac{28}{29} (k-2) \binom{n-2}{k-2}$.
Hence, $28(k-2)< 29 + k^2$, contradicting $4\le k\le 5$.
This completes the proof of Theorem~\ref{thm:main}.
\end{proof}

\section{Concluding Remarks}

The main question arising from our work is whether Theorem~\ref{thm:main} holds for all $n\ge 2k+1$. Proposition~\ref{thm:k=3} confirms this for $k=3$. 
Another question is whether the following generalization of Theorems~\ref{lem_three_sets} and \ref{thm:Frankl96} is true.
We say a family $\mathcal{H}$ of sets has the \emph{EKR property} if the largest intersecting subfamily of $\mathcal{H}$ is trivial.

\begin{conjecture} \label{prob1} 
Suppose $n=n_1+\cdots+n_d$ and $k \ge k_1 + \cdots + k_d$, where $n_i> k_i\ge 0$ are integers.  Let $X_1\cup \cdots \cup X_d$ be a partition of $[n]$ with $|X_i|=n_i$, and
\[
\mathcal{H}:=\left\{F \subseteq \binom{[n]}{k}: |F \cap X_i| \ge k_i~\textrm{for~}i=1, \dots, d \right\}. 
\]
If $n_i \ge 2k_i$ for all $i$ and $n_i > k-\sum_{j=1}^d k_j+k_i$ for all but at most one $i\in [d]$ such that $k_i > 0$, then $\mathcal{H}$ has the EKR property.
\end{conjecture}

The assumptions on $n_i$ cannot be relaxed for the following reasons. 
If $n_i < 2k_i$ for some $i$, then $\mathcal{H}$ itself is intersecting and $|\mathcal{H}(x)| < |\mathcal{H}|$ for any $x\in [n]$.
If $n_i\le k-\sum_{j=1}^d k_j+k_i$ for distinct $i_1, i_2$ such that $k_{i_1}, k_{i_2}>0$, then for any $x\in [n]$, the union of $\mathcal{H}(x)$ and $\{F\in \mathcal{H}: X_{i_1}\subseteq F \text{ or } X_{i_2}\subseteq F\}$ is a larger intersecting family than $\mathcal{H}(x)$.

When $k= k_1 + \cdots + k_d$, Conjecture~\ref{prob1} follows from Theorem~\ref{thm:Frankl96}, in particular, the $d=1$ case is the EKR theorem. 
A recent result of Katona~\cite{Katona17} confirms Conjecture~\ref{prob1} for the case $d=2$ and $n_1, n_2 \ge 9(k-\min\{k_1, k_2\})^2$.
We can prove Conjecture~\ref{prob1} in the following case. 

\begin{theorem} \label{thm:F}
Given positive integers $d\le k$, $2\le t_1\le t_2 \le \cdots \le t_d$ with $t_2\ge k-d+2$,
there exists $n_0$ such that the followings holds for all $n\ge n_0$.
If $T_1, \dots, T_d$ are disjoint subsets of $[n]$ such that $|T_i|= t_i$ for all $i$,
then 
\[
\mathcal{H}:=\left\{F \subseteq \binom{[n]}{k}: |F \cap T_i| \ge 1~\textrm{for~}i=1, \dots, d \right\}
\]
has the EKR property.
\end{theorem}
We omit the proof of Theorem~\ref{thm:F} here because the purpose of this paper is to prove Theorem~\ref{thm:main}.
Moreover, when $d=3$ and $t_1=t_2 = t_3 = k-1$, our $n_0$ is $\Omega(k^4)$ so we cannot replace Theorem~\ref{lem_three_sets} by Theorem~\ref{thm:F} in our main proof.
Nevertheless, it would be interesting to know the smallest $n_0$ such that Theorem~\ref{thm:F} holds.


\end{document}